\documentclass{commat}

\DeclareMathOperator{\aff}{aff}

\DeclareMathOperator{\conv}{conv}

\DeclareMathOperator{\dlat}{d}
\DeclareMathOperator{\flt}{flt}

\DeclareMathOperator{\inte}{int}

\DeclareMathOperator{\LE}{G}
\DeclareMathOperator{\lin}{lin}

\DeclareMathOperator{\sur}{F}

\DeclareMathOperator{\vol}{vol}
\DeclareMathOperator{\V}{V}
\DeclareMathOperator{\width}{w}

\newcommand{\ban}{B_n}

\newcommand{\cKn}{\mathcal{K}_{os}^n}

\newcommand{\dual}[1]{{#1}^\star}

\newcommand{\ip}[2]{\left\langle #1,#2\right\rangle}
\newcommand{\Kcn}{{\mathcal K}_{oc}^n}

\newcommand{\Kn}{{\mathcal K}^n}
\newcommand{\Knull}{{\mathcal K}_{o}^n}

\newcommand{\Lat}{\mathcal{L}}

\newcommand{\N}{\mathbb{N}}

\newcommand{\R}{\mathbb{R}}

\newcommand{\trans}{\intercal}
\newcommand{\va}{{\boldsymbol a}}
\newcommand{\vb}{{\boldsymbol b}}

\newcommand{\ve}{{\boldsymbol e}}

\newcommand{\vnull}{{\boldsymbol 0}}

\newcommand{\vx}{{\boldsymbol x}}
\newcommand{\vy}{{\boldsymbol y}}

\newcommand{\wonull}{\setminus\{\vnull\}}
\newcommand{\Z}{\mathbb{Z}}

\title{%
    Minkowski's successive minima in convex and discrete geometry
    }

\author{%
    Iskander Aliev and Martin Henk
    }

\authorinfo[I. Aliev]{%
   School of Mathematics, Cardiff University, United Kingdom}{%
   alievi@cardiff.ac.uk 
    }

\authorinfo[M. Henk]{%
    Institut f\"ur Mathematik, Technische Universit\"at Berlin,
  Germany}{%
    henk@math.tu-berlin.de
    }

\abstract{%
In this short survey we want to  present some of the impact of Minkowski's successive minima within Convex and Discrete Geometry. Originally related to the volume of an $o$-symmetric convex body,  we point out relations of the successive minima to other functionals, as e.g., the lattice point enumerator or the  intrinsic volumes  and we  present some  old and new conjectures about them. Additionally, we discuss an application of successive minima to a version of Siegel's lemma.
}



\VOLUME{31}
\NUMBER{2}
\firstpage{35}
\DOI{https://doi.org/10.46298/cm.11155}

\begin{document}

\section{Introduction}
One of the basic questions in Geometry of Numbers, as well as in other
areas of mathematics like number theory or integer linear programming,
is to decide when
a set $S$ in the $n$-dimensional Euclidean space $\R^n$ contains an integral
point, i.e., a~point of the lattice $\Z^n$, possibly $\ne\vnull$, and
if necessary to determine such a~point. Here
``when'' usually refers to certain sizes/properties of the
set $S$, such as volume, thickness, covering or packing properties etc.

With respect to the class $\cKn$ of $o$-symmetric convex bodies, i.e.,
non-empty convex and compact sets $K\subset\R^n$ satisfying $K=-K$, and the
volume $\vol(\cdot)$, i.e., the $n$-dimensional Lebesgue measure,
Minkowski's classical, so called \emph{convex body theorem}
gives a~beautiful answer.
\begin{theorem}[Minkowski's convex body theorem,\! 1893,\! \protect{\cite{Minkowski1893}}]\!
 Let $K\in\cKn$ with $\vol(K)\geq 2^n$.  Then $K$ contains a~non-trivial lattice point,
  i.e., $K\cap\Z^n \wonull\ne\emptyset$.
\label{thm:convexbodytheorem}
\end{theorem}
The lower bound on the volume is best possible, as, e.g., the cube
$C_n:=[-1,1]^n$ shows and Minkowski
called any convex body $K\in\cKn$ with $\vol(K)=2^n$ and
$\inte(K)\cap\Z^n =\{\vnull\}$, where
$\inte(\cdot)$ denotes the interior, an \emph{extremal convex
  body}. These are -- up to a~factor of $2$ -- exactly those convex bodies, actually polytopes,
which tile the space via $\Z^n$, i.e., $\Z^n$ is a~covering lattice as well
as a~packing lattice of $\frac{1}{2}K$ (see Section~\ref{sec:preliminaries} for precise definitions).

In his 1896 published book "Geometrie der Zahlen", Minkowski describes this
result as ``\dots ein Satz, der nach meinem Daf\"urhalten zu den
fruchtbarsten in der Zahlenlehre
zu rechnen ist.'' (\cite{Minkowski1896und1910}, p.~75), and indeed this theorem
has numerous applications in different areas for which we refer to~\cite{BombieriGubler2006, Gruber2007, MicciancioGoldwasser2002,
  Schmidt1979, Schmidt1991}.

Actually, in~\cite{Minkowski1893} Minkowski proved Theorem~\ref{thm:convexbodytheorem} as an inequality relating the
volume of $K$ and the minimal norm of a~non-trivial lattice point,
measured with respect to the gauge body $K$. In order to state this
inequality
 -- and later to generalize it -- we will use his 1896
introduced ``kleinstes System von unabh\"angig gerichteten
Strahlendistanzen im Zahlengitter'' (\cite{Minkowski1896und1910}, p.~178), the
so called \emph{successive minima}.

For $K\in\cKn$, $\dim (K)=n$, and $1\leq i\leq n$, the $i$-th
successive minimum $\lambda_i(K)$ is the smallest dilation factor
$\lambda$ such that $\lambda\,K$ contains at least $i$ linearly
independent
lattice points, i.e.,
\begin{equation*}
  \lambda_i(K):=\min\{\lambda >0 : \dim(\lambda K\cap\Z^n)\geq i\}.
\end{equation*}
For instance, if $B = B(a_1, \dotsc, a_n) := [-a_1,a_1]\times\cdots\times [-a_n,a_n]$ is a~box with
\[ a_1 \ge a_2 \ge \dotsb \ge a_n>0,\]
then $\lambda_i(B)=1/a_i$, $1\leq i\leq n$.

The successive minima form a~non-decreasing sequence
and, in particular, $\lambda_1(K)$ is the smallest number $\lambda$ such
that $\lambda\,K$ contains a~non-trivial lattice point.
Hence, Theorem~\ref{thm:convexbodytheorem} is equivalent to
\begin{theorem}
Let $K \in \cKn$, $\dim K=n$. Then
  \begin{equation*}
               \lambda_1(K)^n \,\vol(K)\leq 2^n.
             \end{equation*}
\label{thm:firstminimum}
\end{theorem}
The proof is based on the observation that due to the symmetry of $K$
and the definition of $\lambda_1(K)$, $\Z^n$ is a~packing lattice of $\frac{1}{2}\lambda_1(K)\,K$ and so its volume is at most $1$. Minkowski's so-called theorem
on successive minima is a~far reaching extension of Theorem~\ref{thm:firstminimum} in which $\lambda_1(K)^n$ is replaced by the product of
all successive minima. In addition, in this way a
lower bound is also possible which does not exist in Theorem~\ref{thm:firstminimum}.

\begin{theorem}[Minkowski's theorem on successive minima,\! 1896,\! \protect{\cite[Kapitel 5]{Minkowski1896und1910}}]\!\! Let $K\in\cKn$, $\dim K=n$. Then
 \begin{equation}
               \frac{2^n}{n!}\leq
               \lambda_1(K)\lambda_2(K)\cdots\lambda_n(K)\,\vol(K)\leq
               2^n.
\label{eq:secondminimum}
\end{equation}
\label{thm:secondminimum}
\end{theorem}
The lower bound follows by an inclusion argument and it is attained,
e.g., for the regular cross-polytope
$\dual{C_n}:=\conv(\{\pm\ve_i: 1\leq i\leq n\})$, where $\conv(\cdot)$
denotes the convex hull of a~set, and $\ve_i$ are the canonical unit vectors.
The crucial part is the upper bound, which is considered as a~deep and
important result in Geometry of Numbers. There are many alternative
proofs available for this result (see, e.g.,~\cite[Section 9 and
Section ii]{GruberLekkerkerker1987},~\cite[Section
23]{Gruber2007},~\cite[Section 3.5]{TaoVu2006a} and the references within), but, maybe, the most
geometric one is still Minkowski's original proof (see, e.g.,~\cite[Theorem 23.1]{Gruber2007}).

The applications of the upper bound in Theorem~\ref{thm:secondminimum}
 are probably not as numerous as those of its
precursor Theorem~\ref{thm:firstminimum}, but, in general, they are
also less ``elementary''. For instance, the bound is used in order to prove
Minkowski's finiteness theorems in reduction theory
(see, e.g.,~\cite{Waerden1956}), it appears in the proof of W. M.~Schmidt of his
(strong)
subspace theorem~\cite[pp.~162]{Schmidt1979}
or it is also used in proofs of Freiman's theorem in additive
combinatorics (see, e.g.,~\cite{Chang2002}). Another prominent
application is the Bombieri-Vaaler extension of Siegel's lemma which
we will discuss in more detail in Section~\ref{sec:bombierivaalersiegel}.

Since Minkowski's introduction of the successive minima they have
become an important tool/measure in different areas of mathematics.
In this short survey we want to present some of their impact
within convex and discrete geometry by showing relations of the
successive minima to other functionals (e.g., lattice point
enumerator, intrinsic volumes) as well as presenting some old and new
conjectures related to them. For their immense impact on Diophantine
Geometry we refer to~\cite{BombieriGubler2006, Schmidt1979} and for
algorithmic questions related to them see~\cite{MicciancioGoldwasser2002, Peikert2014a}.

\section{Preliminaries}\label{sec:preliminaries}
Here we briefly introduce some more basic notation, for a~thorough treatment
we refer to~\cite{Cassels1971, GruberLekkerkerker1987, Gruber2007, Schneider2014}.
Let $\Kn$ be the family of all non-empty convex bodies in $\R^n$, i.e.,
compact convex sets $K\subset\R^n$, and let $\Knull\subset\Kn$ be the
set of those convex
bodies with $\vnull\in\inte(K)$. The subfamily of centered convex
bodies, i.e., those $K\in\Knull$ whose centroid
\begin{equation*}
      \frac{1}{\vol(K)}\int _K\vx\, \dlat^n \vx
\end{equation*}
is at the origin is denoted by $\Kcn$. The $n$-dimensional Euclidean
unit ball is denoted by $\ban$ and its volume by $\omega_n$. According to a~result of Steiner
the volume of $K+\rho B_n$, $\rho\geq 0$, is a~polynomial of degree $n$ in $\rho$ which
we can write as~\cite[Section 4.2]{Schneider2014}
\begin{equation*}
  \vol(K+\rho\ban)=\sum_{i=0}^n \omega_i\V_{n-i}(K) \rho^i.
\end{equation*}
The coefficient $\V_i(K)$
is called the $i$-th intrinsic volume of
$K$, $i=0,\dots,n$; in
particular,
\begin{equation}
   \V_n(K)=\vol(K),\quad\V_{n-1}(K)=\frac{1}{2}\sur(K),\quad
   \V_0(K)=1,
\label{eq:intrinsicvolumes}
 \end{equation}
where $\sur(K)$ is the surface area of $K$. We remark that intrinsic
volumes are, up to some normalization, are mixed volumes for which we
refer to~\cite{BuragoZalgaller1988, Schneider2014}.

For $K\in \Knull$ the polar
set $\dual{K}$, defined as
\begin{equation*}
              \dual{K}:=\{\vy\in\R^n : \ip{\vx}{\vy}\leq 1\text{ for
                all }\vx\in K\},
            \end{equation*}
is again a~convex body. Here $\ip{\cdot}{\cdot}$ denotes the standard
inner product on $\R^n$.
The set of all $m$-dimensional
lattices $\Lambda\subset\R^n$ will be denoted by $\Lat^n_m$, i.e.,
\begin{equation*}
  \Lat^n_m:=\{B\Z^m : B\in\R^{n\times m},\,\text{\rm rank}(B)=m\}.
\end{equation*}
In the case $m=n$ we just write $\Lat^n$.
As usual, for $\Lambda=B\Z^m \in\Lat^n_m$,
$\det(\Lambda):=\sqrt{\det(B^\trans B)}$ is called the determinant of $\Lambda$.
The polar lattice of $\Lambda$ is given by
\begin{equation*}
   \dual{\Lambda}:=\{\vy\in \lin(\Lambda): \ip{\vx}{\vy}\in\Z\text{
     for all }\vx\in\Lambda\},
\end{equation*}
and it is $\det(\dual{\Lambda})=1/\det(\Lambda)$. Here $\lin(\cdot)$
denotes the linear hull.

With respect to a~convex body $K\in\Knull$ and a~general lattice $\Lambda\in\Lat^n$, the $i$-th successive
minimum $\lambda_i(K,\Lambda)$ is given by
\begin{equation*}
  \lambda_i(K,\Lambda):=\min\{\lambda >0 : \dim(\lambda K\cap\Lambda)\geq i\},
\end{equation*}
where the dimension of a~set $S$ is always meant with respect to its
affine hull $\aff(S)$.
In case $\Lambda=\Z^n$ we write $\lambda_i(K)$ and if
$K=B_n$, we abbreviate
$\lambda_i(B_n,\Lambda)$ by $\lambda_i(\Lambda)$, which is also called
the $i$-th successive minimum of the lattice $\Lambda$.
Since for
$\Lambda=B\Z^n \in\Lat^n$ we have $\lambda_i(B^{-1}K)=\lambda_i(K,\Lambda)$
and $\vol(B^{-1}K)= \vol(K)/\det(\Lambda)$, Minkowski's
Theorem~\ref{thm:secondminimum} can be equivalently stated for arbitrary lattices
and $K\in\cKn$ as
\begin{equation}
\det(\Lambda)\frac{2^n}{n!}\leq
\lambda_1(K,\Lambda)\lambda_2(K,\Lambda)\cdots\lambda_n(K,\Lambda)\,\vol(K)\leq
2^n \det(\Lambda).
\label{eq:secondminimumgenlat}
\end{equation}

A lattice $\Lambda\in\Lat^n$ will be called a~covering lattice of
$K\in\Kn$ if $\R^n =\Lambda+K$ and a~packing lattice if
$\inte(\va+K)\cap\inte(\vb+K)=\emptyset$ for all $\va\ne\vb\in
\Lambda$. Given a~$K\in \Kn$ and a~lattice $\Lambda\in\Lat^n$, then
\begin{equation*}
  \lambda_1(K-K,\Lambda)=\max\{\rho: \Lambda \text{ packing lattice of
  } \rho\,K\}
\end{equation*}
and
\begin{equation}
  \delta(K):=\max\left\{\frac{\vol\big(\lambda_1(K-K,\Lambda)
        \,K\big)}{\det(\Lambda)}: \Lambda\in\Lat^n \right\}
\label{def:density}
\end{equation}
is called the density of a~densest lattice packing of $K$. Here the
ratio \[\vol(\lambda_1(K-K,\Lambda)\,K)/\det(\Lambda)\]
describes the proportion of space which is occupied by the
packing $\Lambda+\lambda_1(K-K)\,K$.
Observe, that $0<\delta(K)\leq 1$.

The covering counterpart to $\lambda_1(K-K,\Lambda)$ is the so called
covering radius
\begin{equation*}
    \mu(K,\Lambda)=\min\{\mu>0 : \mu\,K+\Lambda=\R^n \},
\end{equation*}
i.e., the smallest $\mu>0$ such that $\Lambda$ is a~covering lattice of
$\mu\,K$. In analogy to~\eqref{def:density}, the minimum of $\vol(\mu(K,\Lambda)\,K)/\det(\Lambda)$
with respect to all lattices leads to the density of a~thinnest lattice
covering, but we do not need this quantity here.

By definition we have
$\lambda_1(K-K,\Lambda)\leq \mu(K,\Lambda)$ and more generally,
\begin{equation*}
  \lambda_n(K-K,\Lambda)\leq\mu(K,\Lambda)\leq
     \lambda_1(K-K,\Lambda)+\cdots + \lambda_n(K-K,\Lambda).
\end{equation*}
This was shown in the symmetric case by Jarn{\'\i}k (see~\cite[Theorem 23.4]{Gruber2007}). The general case
was treated
by Kannan and Lov\'asz in~\cite[Lemma 2.4]{KannanLovasz1986}, where they
also introduced the so called
covering minima $\mu_i(K,\Lambda)$, $1\leq i\leq n$, which can be
regarded as covering counterparts to the successive minima. For more
information on these functionals see,e.g.,~\cite{KannanLovasz1986, MerinoSchymura2017}.

\section{Possible Tightenings and Generalizations of Minkowski's theorem}
First we state a~straightforward extension of Minkowski's Theorem~\ref{thm:secondminimum} to arbitrary, not necessarily $o$-symmetric convex bodies. To this
end we consider for $K\in\Kn$ its central symmetrical
\begin{equation*}
       K_s:=\frac{1}{2}(K-K)\in\cKn.
     \end{equation*}
 Obviously, $K_s=K$ for $K\in\cKn$ and by the classical Brunn-Minkowski inequality~\cite[Section
     7]{Schneider2014} we know that
     \begin{equation*}
      \vol(K)\leq \vol(K_s)
     \end{equation*}
with equality if and only if $K$ and $K_s$ are translates of each
other. Thus, the upper bound in~\eqref{eq:secondminimumgenlat} applied to
$K_s$ gives
\begin{equation*}
\lambda_1(K_s,\Lambda)\lambda_2(K_s,\Lambda)\cdots\lambda_n(K_s,\Lambda)\,\vol(K)\leq 2^n \det(\Lambda),
\end{equation*}
and it is also easy to see that the corresponding lower bound in~\eqref{eq:secondminimumgenlat} holds true (see, e.g., the discussion in~\cite{HenkHenzeHernandezCifre2016}). Hence, Minkowski's Theorem~\ref{thm:secondminimum} can be stated in a~bit more general form.
\begin{theorem}
Let
  $K\in\Kn$, $\dim K=n$, and let $\Lambda\in\Lat^n$. Then
\begin{equation}
\det(\Lambda)\frac{2^n}{n!}\leq
\lambda_1(K_s,\Lambda)\lambda_2(K_s,\Lambda)\cdots\lambda_n(K_s,\Lambda)\,\vol(K)\leq
2^n \det(\Lambda).
\label{eq:secondminimumgen}
\end{equation}
\label{thm:secondminimumgen}
\end{theorem}
In fact, the body $K_s$ is a~rather natural candidate for such an
extension. As $\Lambda$ is a~packing lattice of
$\lambda_1(K-K,\Lambda)\,K$ we get by the definition of the density of
a densest lattice packing~\eqref{def:density} the inequality
\begin{equation*}
  \lambda_1(K_s,\Lambda)^n \vol(K) =
  2^n \lambda_1(K-K,\Lambda)^n \vol(K)\leq \delta(K)\,2^n \det(\Lambda)
\end{equation*}
which is a~tightening of Minkowski's Theorem~\ref{thm:firstminimum}
(for general $K\in\Kn$ and $\Lambda\in\Lat^n$). If such an
improvement is also possible for the upper bound in Minkowski's Theorem~\ref{thm:firstminimum} is the content of a~famous problem posed by
Davenport.

\begin{problem}[Davenport, 1946, \protect{\cite{Davenport1946}}]
  Let $K\in\Kn$. Is it true that
\begin{equation}\label{eq:conjecture}
    \lambda_1(K_s,\Lambda)\cdots\lambda_n(K_s,\Lambda)\vol(K)\leq
    \delta(K)\, 2^n \det(\Lambda)
    \,\hbox{\rm\Large?}
\end{equation}
\label{conj:davenport}
\end{problem}
Actually, Davenport formulated it for $K\in\cKn$ but this is
equivalent to the above statement.
So far it has only been verified for $n=2$ and for ellipsoids by
Minkowski (see~\cite[pp.~196]{Minkowski1896und1910},~\cite[pp.~195]{GruberLekkerkerker1987}),
the case $n=3$ has been settled by Woods~\cite{Woods1956}. For more
information see~\cite[Section 18]{GruberLekkerkerker1987}.

The inequalities in Theorem~\ref{eq:secondminimumgen} have the nice
properties that they generalize the symmetric setting
and they are invariant with respect to translations of
$K$. A different way to extend the class $\cKn$ is to consider
centered convex bodies $K\in\Kcn$.
For those bodies Ehrhart posed in 1964 the following conjecture as
an analog to Minkowski's convex body Theorem~\ref{thm:convexbodytheorem}.
\begin{conjecture}[Ehrhart, 1964,~\cite{Ehrhart1964}]
 Let $K\in\Kcn$
  with $\vol(K)\geq (n+1)^n /n!$. Then $K$ contains a~non-trivial
  lattice point, i.e., $K\cap\Z^n \wonull\ne\emptyset$.
\label{conj:ehrhart}
\end{conjecture}
Moreover, he conjectured that the bound is best possible only (up to
$\Z^n$ preserving linear transformations) for the
simplex
\begin{equation}
  T_n:=-\sum_{i=1}^n \ve_i +(n+1)\conv\{\vnull, \ve_1, \dotsc, \ve_n\}.
\label{eq:centeredsimplex}
\end{equation}
Ehrhart proved his conjecture in various special cases, e.g., in the
plane~\cite{Ehrhart1955a} and for
simplices~\cite{Ehrhart1979a}. Berman and Berndtsson~\cite{BermanBerndtsson2017} proved it for a~special class
of $n$-dimensional lattice polytopes including so called reflexive polytopes (see also~\cite{NillPaffenholz2014}). The best known bound for which the conclusion of
Conjecture~\ref{conj:ehrhart} holds true is based on a~lower bound for
the ratio
\begin{equation*}
         \alpha(K):=\frac{\vol(K\cap(-K))}{\vol(K)}, \quad K\in\Kcn,
\end{equation*}
since then $\vol(K)\geq \alpha(K)^{-1} 2^n$ implies by Minkowski's
Theorem~\ref{thm:convexbodytheorem} that $K\cap(-K)$ and thus $K$ contains a~non-trivial
lattice point.

Milman and Pajor~\cite{MilmanPajor2000} proved $\alpha(K)\geq
2^{-n}$, this was improved by Huang et
al.~\cite{HuangSlomkaTkoczEtAl2022} to
$2^{-n}\mathrm{e}^{c\sqrt{n}}$ where $c$ is an absolute constant, and
recently it was shown by Campos et al.~\cite[Theorem 4.1]{CamposHintumMorrisEtAl2022} that
\begin{equation*}
  \alpha(K)\geq 2^{-n}\,\mathrm{e}^{c\,n/L_n^2},
\end{equation*}
where $L_n$ is the so called isotropic constant (see, e.g.,~\cite{BrazitikosGiannopoulosValettasEtAl2014}). Together
with the very recently announced upper bound of $O(\sqrt{\log(n)})$
onto $L_n$ by Klartag~\cite{Klartag2023}, the result of Campos et al.~shows
that $K\in\Kcn$
contains a~non-trivial lattice point if $\vol(K)\geq 4^n
\mathrm{e}^{-c\,n/\log(n)^2} $ (cf.~\cite[Theorem 1.4]{CamposHintumMorrisEtAl2022}); observe that $(n+1)^n /n! \sim \mathrm{e}^n$
and it is believed that $\alpha(K)$ is
minimized for a~simplex which would give ``almost'' a~bound of $\mathrm{e}^n$.

In view of Minkowski's successive minima it is also tempting to
consider the following generalization of Ehrhart's conjecture (see~\cite{HenkHenzeHernandezCifre2016}).
\begin{problem}
Let $K\in \Kcn$ and $\Lambda\in\Lat^n$. Is it true
  that
  \begin{equation*}
     \lambda_1(K,\Lambda)\cdots\lambda_n(K,\Lambda)\vol(K)\leq \frac{(n+1)^n}{n!}\det(\Lambda) \,\hbox{\rm\Large?}
  \end{equation*}
\end{problem}
By the same reasoning as above and the monotonicity of the successive
minima such an inequality exists with $4^n
\mathrm{e}^{-c\sqrt{n}}$ instead of $(n+1)^n /n!$, and in~\cite{HenkHenzeHernandezCifre2016} it is
also verified in some special cases, e.g., in the plane. Moreover, in
analogy to Minkowski's lower bound in Theorem~\ref{thm:secondminimum}, the
authors of \cite{HenkHenzeHernandezCifre2016} prove that for $K\in\Kcn$,
\begin{equation*}
           \frac{n+1}{n!}\det(\Lambda)\leq \lambda_1(K,\Lambda)\cdots\lambda_n(K,\Lambda)\vol(K)
\end{equation*}
along with a~characterization of the equality case. In particular, if
$\Lambda=\Z^n$,
equality is attained for
\begin{equation}
            \dual{T_n}=\conv\{-(\ve_1+\cdots +\ve_n),
            \ve_1,\dots,\ve_n \}.
\label{eq:dualtn}
\end{equation}

As the volume is a~particular intrinsic volume $\V_i(K)$ (see~\eqref{eq:intrinsicvolumes}) one may also
ask for inequalities relating these functionals to the successive
minima. The $\V_i(K)$, however, are not
invariant under linear transformations of determinant $1$ and so
results for $\Z^n$ cannot be equivalently formulated for arbitrary
lattices as in the case of the volume inequalities presented so far.

As the $i$-th intrinsic volume is not smaller than the volume of any
intersection of $K$ with an $i$-dimensional plane, the lower bound in Theorem~\ref{thm:secondminimumgen} implies for $K\in\Kn$, $\dim K=n$, and $\Lambda=\Z^n$
(see~\cite{Wills1990b})
\begin{equation}
         \frac{2^i}{i!}\leq \lambda_1(K_s)\cdots\lambda_i(K_s)
         \V_i(K),\quad 1\leq i\leq n.
\label{eq:wills}
\end{equation}
An $o$-symmetric convex body without non-trivial lattice
points, but with a~large $i$-di\-men\-sion\-al section shows there is no
upper bound on the right hand side for $i\leq n-1$. On the other hand,
it was shown in~\cite{Henk1990} that for $K\in\Kn$, $\dim K=n$,
\begin{equation}
       \lambda_{i+1}(K_s)\cdots\lambda_n(K_s)
       < 2^{n-i} \frac{\V_i(K)}{\vol(K)},\quad 1\leq i\leq n-1,
\label{eq:henk}
\end{equation}
and that this inequality is best possible. Actually, both inequalities~\eqref{eq:wills} and~\eqref{eq:henk} were originally proved for
$o$-symmetric convex bodies, but the proofs also work in this slightly
more general setting. We also remark, that in general the ratio $\V_i(K)/\vol(K)$
 in~\eqref{eq:henk} is not bounded from above in terms of the
 successive minima. To generalize~\eqref{eq:henk} to the setting of arbitrary lattices is still an open
problem. 
\begin{conjecture}[Schnell, 1995,~\cite{Schnell1995a}]
 Let $K\in\Kn$,
  $\dim K=n$, and let $\Lambda\in\Lat^n$. Then
\begin{equation*}
       \lambda_{i+1}(K_s,\Lambda)\cdots\lambda_n(K_s,\Lambda)
       < 2^{n-i} \frac{\V_i(K)}{\vol(K)} \frac{\det(\Lambda)}{\det_i(\Lambda)},\quad 1\leq i\leq n-1,
\end{equation*}
where $\det_i(\Lambda)$ is the smallest determinant of an $i$-dimensional
sublattice of $\Lambda$.
\end{conjecture}
Schnell made this conjecture only for $K\in\cKn$, but in view of~\eqref{eq:henk} it is quite plausible that it holds true for any
$K\in\Kn$.

The case $i=n-1$ in~\eqref{eq:henk} gives the particularly nice
inequality
\begin{equation}
  \lambda_n(K_s)<\frac{\sur(K)}{\vol(K)},
\label{eq:survol}
\end{equation}
where $\sur(K)$ is the surface area of $K$ (cf.~\eqref{eq:intrinsicvolumes}). In~\eqref{eq:discretevolsur} we will also see a~kind of discrete counterpart
to~\eqref{eq:survol}.

For $K\in \cKn$ and $i=n-1$,~\eqref{eq:wills} was improved in~\cite{HenkHenzeHernandezCifre2016} to the tight inequality
\begin{equation*}
   \frac{2^n}{(n-1)!}\leq \left(\sum_{i=1}^n \prod_{j=1,\, j\ne i}^n \lambda_j(K)^2 \right)^\frac{1}{2} \sur(K).
\end{equation*}
A corresponding best possible inequality for $K\in\Kn$, or with respect to arbitrary
lattices or for other intrinsic volumes is not known (see the
discussion in~\cite{HenkHenzeHernandezCifre2016}).

\section{Successive Minima and polarity}
In the context of so-called transference theorems in number theory the
goal is (roughly speaking) to establish relations between
integral solutions of different linear Diophantine approximation
problems (see, e.g.,~\cite[Section 45]{GruberLekkerkerker1987},~\cite[Section 23.2]{Gruber2007}).
From a~geometric point of view, this means
to relate functionals such as volume or successive minima of $K$
and $\dual{K}$. The study of this fruitful interplay
  goes back to Kurt Mahler in 1930s (see, e.g,~\cite{Evertse2019}),
and in~\cite{Mahler1939} he studied for $K\in\cKn$ the linearly
invariant volume product
$\vol(K)\vol(\dual{K})$, today also known as Mahler volume.
\begin{conjecture}[Mahler, 1939,~\cite{Mahler1939}]
 Let
  $K\in\cKn$. Then
\begin{equation}
   \frac{4^n}{n!}=\vol(C_n)\vol(\dual{C_n})\leq \vol(K)\vol(\dual{K}).
 \end{equation}
 \label{conj:mahler}
\end{conjecture}
Mahler~\cite{Mahler1938} verified the conjecture in dimension 2, and
it was also recently proved in dimension 3~\cite{IriyehShibata2020}. He also proved a~lower bound of $4^n /(n!)^2$
and an upper bound on the product of $4^n$. The best known general
bounds are
\begin{equation}
  \frac{\pi^n}{n!}\leq\vol(K)\vol(\dual{K})\leq \omega_n^2.
\label{eq:mahler_volume}
\end{equation}
The apparently best possible upper bound, which was also
conjectured by Mahler~\cite{Mahler1939}, is known as the Blaschke-Santal\'o Theorem~\cite{Santalo1949}. The lower bound is due to
Kuperberg~\cite{Kuperberg2008}. In the general case, i.e., for $K\in\Kn$ it is conjectured, and also
attributed to Mahler, that
\begin{equation}
  \frac{(n+1)^{n+1}}{(n!)^2}= \vol(S_n)\vol(\dual{S_n})\leq \vol(K)\vol(\dual{K}),
\label{eq:mahler2}
\end{equation}
where $S_n$ is any simplex with the centroid at the origin, e.g.,
$T_n$ from~\eqref{eq:centeredsimplex}. This is only known to be true
in the plane~\cite{Mahler1938}. For more information on the Mahler
volume and its central role within Convex Geometry we refer to~\cite{FradeliziMeyerZvavitch2023}.

Combining the upper bound in Minkowski's inequality~\eqref{eq:secondminimumgen} for
$\dual{K}$ with the conjectured
lower bound in Mahler's Conjecture~\eqref{conj:mahler} leads for $K\in\cKn$ to
the following conjectural inequality, which is also due to Mahler.
\begin{conjecture}[Mahler, 1974,~\cite{Mahler1974}]
 Let
  $K\in\cKn$ and $\Lambda\in\Lat^n$. Then
\begin{equation}
  \frac{2^n}{n!}\det(\Lambda) \,\,\lambda_1(\dual{K},\dual{\Lambda})\cdots \lambda_n(\dual{K},\dual{\Lambda})
  \leq \vol(K).
\label{conj:mahler_2}
\end{equation}
\end{conjecture}
The inequality would be best possible, for instance, for the
cross-polytope $\dual{C_n}$, and the previously mentioned results on the volume product imply that
it is true for $n=2,3$. Even the weaker inequality (as an analogue to
Theorem~\ref{thm:firstminimum}),
\begin{equation}
   \frac{2^n}{n!}\det(\Lambda)\,\, \lambda_1(\dual{K},\dual{\Lambda})^n
   \leq \vol(K),
\label{eq:mahler_weaker}
\end{equation}
which has also been studied by Mahler, is open for $n\geq 4$.
For general convex bodies the same problem was studied by
Makai Jr.
\begin{conjecture}[Makai, Jr., 1978,~\cite{Makai1978}]
 Let $K\in\Kn$ and
  $\Lambda\in\Lat^n$. Then
  \begin{equation}
      \frac{n+1}{n!}\det(\Lambda)\,\,
      \lambda_1(\dual{K_s},\dual{\Lambda})^n \leq \vol(K).
 \label{eq:makai}
\end{equation}
\label{conj:makai}
\end{conjecture}
It was shown to be true for $n=2$ by L. Fejes T\'oth and Makai, Jr.~\cite{FejesTothMakai1974} (see also~\cite{MerinoSchymura2017} for applications). In view of~\eqref{conj:mahler_2}
one might even conjecture the stronger inequality
\begin{equation}
   \frac{n+1}{n!}\det(\Lambda)
   \lambda_1(\dual{K_s},\dual{\Lambda})\cdots
   \lambda_n(\dual{K_s},\dual{\Lambda})\leq \vol(K),
\label{eq:makai_2}
\end{equation}
which would be also best possible as the simplex
$\dual{T_n}$ from~\eqref{eq:dualtn} shows. For $n=2$,~\eqref{eq:makai_2}
is an immediate consequence of the upper bound in~\eqref{eq:secondminimumgen} and Eggleston's~\cite{Eggleston1961} inequality
\begin{equation}
 6\leq \vol(K)\vol(\dual{K_s})
  \label{eq:eggelston}
\end{equation}
for planar convex bodies $K$. Actually, taking into account all
successive minima, there seems to be stronger lower bounds possible than the one
in~\eqref{eq:makai_2}. For the planar case see~\cite{HenkXue2019}. In contrast to Minkowski's Theorem~\ref{thm:secondminimumgen}, here
upper bounds on $\vol(K)$ in terms of
$\lambda_i(\dual{K},\dual{\Lambda})$ are the easy part. In~\cite{HenkXue2019} it was shown that for $K\in \Kn$
\begin{equation*}
   \vol(K)\leq 2^n \det(\Lambda)\lambda_1(\dual{K_s},\dual{\Lambda})\cdots\lambda_n(\dual{K_s},\dual{\Lambda}),
 \end{equation*}
and if $K \in\Kcn$ then
\begin{equation*}
   \vol(K)\leq \frac{(n+1)^n}{n!} \det(\Lambda)\lambda_1(\dual{K_s},\dual{\Lambda})\cdots\lambda_n(\dual{K_s},\dual{\Lambda}).
\end{equation*}
Both inequalities are best possible as the usual suspects show.

The
inequalities~\eqref{eq:mahler_weaker} and~\eqref{eq:makai} have a~nice
geometric interpretation in terms of the so called lattice width. For $K\in\Kn$ we
have by the definition of the polar set that
\begin{equation}
  \lambda_1(\dual{(K-K)},\dual{\Lambda})=\min_{\dual{\va}\in\dual{\Lambda}\wonull}\,\max_{\vy\in
    K-K}\ip{\dual{\va}}{\vy} =: \width(K,\Lambda),
\label{eq:lattice_width}
\end{equation}
and $\width(K,\Lambda)$ is the lattice width of $K$
with respect to $\Lambda$. It describes, roughly speaking, the minimal
number of parallel lattice hyperplanes of the lattice $\Lambda$
intersecting $K$; a~lattice hyperplane is a~hyperplane containing $n$ affinely independent lattice points of
$\Lambda$.
Hence,~\eqref{eq:mahler_weaker} and~\eqref{eq:makai} claim that the
volume of a
convex body of lattice width $2$ is at least the volume of $\dual{C_n}$
if $K$ is symmetric and, otherwise, at least the volume of $\dual{T_n}$.

An interesting weaker inequality than~\eqref{eq:makai} was
conjectured in the context of isosystolic inequalities for optical hypersurfaces.
\begin{conjecture}[\'Alvarez Paiva et al., 2016,~\cite{PaivaBalacheffTzanev2016}]
 Let $K\in\Knull$ and $\Lambda\in\Lat^n$. Then
\begin{equation}
   \frac{n+1}{n!}\det(\Lambda)\,\,
   \lambda_1(\dual{K},\dual{\Lambda})^n \leq \vol(K),
\label{conj:alvarez}
\end{equation}
with equality if and only if $K$ is a~simplex whose vertices are its
only non-trivial lattice points.
\end{conjecture}
It was pointed out in~\cite{HenkXue2019} that in general there is no upper bound
on $\vol(K)$ in this setting. For possible extensions of Makai's conjecture~\eqref{eq:makai} via
 covering minima (instead of the successive minima) we refer to Gonz\'alez Merino and Schymura~\cite{MerinoSchymura2017}.

Applying Minkowski's upper bound in~\eqref{eq:secondminimumgen} to
$K,\Lambda$ and $\dual{K},\dual{\Lambda}$ gives with~\eqref{eq:mahler_volume}
  \begin{equation*}
           \lambda_1(K,\Lambda)\lambda_1(\dual{K},\dual{\Lambda})\cdots
           \lambda_n(K,\Lambda)\lambda_n(\dual{K},\dual{\Lambda}) \leq \frac{4^n}{\pi^n}n!
  \end{equation*}
  and, in particular,
  \begin{equation*}
               \lambda_1(K,\Lambda)\lambda_1(\dual{K},\dual{\Lambda})\leq c\,n,
  \end{equation*}
where, in the following $c$ denotes an absolute constant which may
vary from line to line. This easily derived inequality is essentially best
possible, as Banaszczyk~\cite[Lemma 2]{Banaszczyk1996} proved
that in every dimension $n$ there exist a~lattice $\Lambda\in\Lat^n$ with
\begin{equation}
                  c\,n\leq \lambda_1(K,\Lambda)\lambda_1(\dual{K},\dual{\Lambda}).
\label{eq:bana_lower}
                \end{equation}
In the case $K=\ban$ this was shown earlier by Conway and Thompson~\cite[Chapter II, Theorem 9.5]{MilnorHusemoller1973}.

For products of the type
$\lambda_i(K,\Lambda)\lambda_j(\dual{K},\dual{\Lambda})$ one can only
expect a~non-trivial lower bound if $j\geq n+1-i$ and an upper bound
if $j\leq n+1-i$.

In the already mentioned paper~\cite{Mahler1939}, Mahler was the
first who studied the products
$\lambda_i(K,\Lambda)\lambda_{n+1-i}(\dual{K},\dual{\Lambda})$
and proved for $K\in\cKn$ the bounds
\begin{equation*}
             1\leq
             \lambda_i(K,\Lambda)\lambda_{n+1-i}(\dual{K},\dual{\Lambda})\leq n!.
\end{equation*}
The lower bound is clearly optimal, but the upper bound
has been improved considerably in the last decades. The currently best bounds are due to
Banaszczyk and are based on his groundbreaking Gaussian-like measures
on lattices introduced in~\cite{Banaszczyk1993a}. In~\cite{Banaszczyk1996} he proves for $K\in\cKn$
\begin{equation*}
    \lambda_i(K,\Lambda)\lambda_{n+1-i}(\dual{K},\dual{\Lambda})\leq
    c\,n (1+\log(n)),
  \end{equation*}
  which, in view of~\eqref{eq:bana_lower} is close to optimal. Moreover, he also
  shows that the $(1+\log(n))$ term can be improved for various classes
  of symmetric convex bodies. In particular, for $K=\ban$ it was
  already shown in~\cite{Banaszczyk1993a} that
  \begin{equation*}
            \lambda_i(\Lambda)\lambda_{n+1-i}(\dual{\Lambda})\leq n.
  \end{equation*}

If $K\in\Kn$ is lattice point free with respect to a~lattice
$\Lambda$. i.e., $\inte(K)\cap\Lambda=\emptyset$, then the covering
radius is at least 1, i.e.,
$\mu(K,\Lambda)\geq 1$. Hence, any upper bound on
$\mu(K,\Lambda)\lambda_1(\dual{K_s},\dual{\Lambda})$, $K\in\Kn$, is a~bound on the
so-called flatness constant $\flt(n)$, the maximal
lattice width of a lattice point free convex body in $\R^n$.

That
this quantity can be indeed bounded by a~constant only depending on
the dimension was first shown by Khinchin~\cite{Khinchin1948}.

For $K\in\cKn$, Banaszczyk~\cite{Banaszczyk1996} proved
\begin{equation}
               \mu(K,\Lambda)\lambda_1(\dual{K},\dual{\Lambda}) \leq c\,n\,\log(n)
\label{eq:flatsymmetric}
             \end{equation}
and so $\flt(n)\leq c\,n\log(n)$ for all convex bodies having a
center of symmetry. For general $K\in\Kn$, the following bound was
very recently announced by Reis and Rothvoss~\cite{ReisRothvoss2023}
(see also~\cite{Rudelson2000a},~\cite{BanaszczykLitvakPajorEtAl1999}
for the former best known bounds)
\begin{equation*}
  \flt(n)\leq c\,n\,\log(n)^8.
\end{equation*}
Actually, their main result implies the astonishing relation
\begin{equation*}
     \mu(K,\Lambda)\leq c\,\log(n)^7 \mu(K-K,\Lambda),
   \end{equation*}
which gives the bound on the flatness constant via the symmetric case~\eqref{eq:flatsymmetric}.  The best lower bounds on $\flt(n)$ are of order $n$ (see~\cite{BanaszczykLitvakPajorEtAl1999, MayrhoferSchadeWeltge2022} and the references within).

\section{Successive Minima and lattice point enumerator}

Since the successive minima measure or reflect
lattice point properties of a~convex body one may also
ask for direct relations between the successive minima
and the lattice point enumerator
$\LE(K,\Lambda):=\#(K\cap\Lambda)$; in the case $\Lambda=\Z^n$ we just
 write $\LE(K)$.
A result in this spirit is again due to Minkowski who proved
in analogy to his convex body Theorem~\ref{thm:convexbodytheorem}.
\begin{theorem}[Minkowski, 1896, \protect{\cite[p.~79]
{Minkowski1896und1910}}] Let
  $K\in\cKn$, $\dim K=n$ and $\LE(K)\geq 3^n +1$. Then
  $K$ contains a~non-trivial lattice point in its interior.
\end{theorem}
The cube $C_n$ shows that the bound cannot be improved in
general. Minkowski also proved a~sharper bound of $2^{n+1}-1$ for the
class of strictly $o$-symmetric convex bodies, but for simplification
we will deal only with the general case and we
refer to~\cite[p.~63]{GruberLekkerkerker1987} for
details, and to~\cite{GonzalezMerinoHenze2016} for an interesting generalization of these statements
of Minkowski.

Betke et al.~\cite{BetkeHenkWills1993} embedded the above result of Minkowski
in an inequality for $o$-symmetric convex bodies in the sense of
Theorem~\ref{thm:firstminimum} which was
later extended to $K\in \Kn$ by Malikiosis~\cite{Malikiosis2012}.
 Let $K\in\Kn$ and
  $\Lambda\in\Lat^n$. Then
\begin{equation*}
        \LE(K,\Lambda) \leq \left\lfloor\frac{2}{\lambda_1(K_s,\Lambda)}+1\right\rfloor^n.
\end{equation*}
And, obviously, the conjecture is that this can also be improved via
the product of all successive minima as in Minkowski's Theorem~\ref{thm:secondminimum} on
successive minima.
\begin{conjecture}
[Betke et al.\protect{\cite{BetkeHenkWills1993}}; Malikiosis
  \protect{\cite{Malikiosis2012}}] Let $K\in\Kn$ and
  $\Lambda\in\Lat^n$. Then
\begin{equation*}
        \LE(K,\Lambda) \leq
        \left\lfloor\frac{2}{\lambda_1(K_s,\Lambda)}+1\right\rfloor\cdots
        \left\lfloor\frac{2}{\lambda_n(K_s,\Lambda)}+1\right\rfloor.
      \end{equation*}
\label{conj:discreteMinkowski}
\end{conjecture}
The cube, or more generally, a~box $[-l_1,l_1]\times \cdots\times
[-l_n,l_n]$, $l_i\in\N$, shows that the bound would be tight.
Again Betke et al. just considered the symmetric case, in which they
also proved a~best possible lower bound~\cite[Corollary
2.1]{BetkeHenkWills1993} if $\lambda_n(K,\Lambda)\leq 2$:
\begin{equation*}
  \frac{1}{n!}\left(\frac{2}{\lambda_1(K,\Lambda)}-1\right)\cdots \left(\frac{2}{\lambda_n(K,\Lambda)}-1\right)\le\LE(K,\Lambda).
\end{equation*}
In~\cite{Malikiosis2012}, Malikiosis proved the Conjecture~\ref{conj:discreteMinkowski} for $n=3$, and in general he showed that
\begin{equation}
        \LE(K,\Lambda) \leq \frac{4}{\mathrm{e}}\sqrt{3}^{n-1}
        \left\lfloor\frac{2}{\lambda_1(K_s,\Lambda)}+1\right\rfloor\cdots
        \left\lfloor\frac{2}{\lambda_n(K_s,\Lambda)}+1\right\rfloor,
\label{eq:malikiosis}
      \end{equation}
where $\sqrt{3}$ can be replaced by $\sqrt[3]{40/9}$ if $K=-K$.
Moreover, in~\cite{Malikiosis2013} he
verified it for ellipsoids in every dimension.

Recently, Tointon~\cite{Tointon2023} presented a
  different type of upper bound on $\LE(K,\Lambda)$ in terms of the
  successive minima:
  \begin{equation*}
       \LE(K,\Lambda) \leq
       \left(1+\frac{\lambda_k(K,\Lambda)}{2}\right)\frac{2}{\lambda_1(K,\Lambda)}\cdots
       \frac{2}{\lambda_k(K,\Lambda)},
  \end{equation*}
where the index $k$ is chosen such that $k=\max\{j:
\lambda_j(K,\Lambda)\leq 2\}$ which can be replaced in the symmetric
setting by $k=\max\{j:
\lambda_j(K,\Lambda)\leq 1\}$. Tointon's inequality improves on~\eqref{eq:malikiosis} in the symmetric case, as well as in various
cases for general $K\in\Kn$. Moreover, it has, as well as Conjecture~\ref{conj:discreteMinkowski}, the nice and important
feature that it implies the continuous case, i.e., the upper bound in Minkowski's
Theorem~\ref{thm:secondminimumgen}. The reason is that, roughly speaking,
for ``fat'' convex bodies there is almost no difference between
$\vol(K)$ and $\LE(K)$, or, more precisely, the Jordan measurability
of convex bodies gives
\begin{equation}
       \lim_{\rho\to\infty} \frac{\vol(\rho\,K)}{\det(\Lambda)
         \LE(\rho\,K,\Lambda)}=1.
\label{eq:vollattice}
\end{equation}

In order to control the gap between $\vol(K)$ and $\LE(K)$ for
``thin'' convex bodies, Betke et al.~also started to study bounds on
$\LE(K)/\vol(K)$ in terms of the successive minima. And here the
following inequalities could be true.

\begin{conjecture}[Betke et al., 1993]
  Let $K\in\Kn$, $\dim(K)=n$ and $\Lambda\in\Lat^n$. Then
  \begin{equation*}
 \prod_{i=1}^n
 \left(1-i\frac{\lambda_i(K_s,\Lambda)}{2}\right) \leq
 \frac{\LE(K,\Lambda)}{\vol(K)}\det(\Lambda)\leq \prod_{i=1}^n \left(1+i\frac{\lambda_i(K_s,\Lambda)}{2}\right),
\end{equation*}
where for the lower bound $n\,\lambda_n(K_s,\Lambda)\geq 2$ is
assumed and $\LE(K)$ might be replaced by $\LE(\inte(K))$.
\label{conj:gv}
\end{conjecture}
Actually, in~\cite[Conjecture 2.2]{BetkeHenkWills1993} Betke et
al. state only a~conjecture about a~corresponding lower bound for
symmetric convex bodies in which the $i$s in the factors of the
product are replaced by $1$, and they pose the problem to consider
also upper bounds.

The bounds in the Conjecture~\ref{conj:gv} would be tight as, e.g.,
positive integral multiples of the standard simplex
$\conv\{\vnull,\ve_1,\dots,\ve_n\}$ show (see~\cite{FreyerLucas2022}). Freyer and Lucas verified in~\cite{FreyerLucas2022}
the upper bound in Conjecture~\ref{conj:gv} in the plane and the lower
bound for planar lattice polytopes. In arbitrary dimensions they
proved the following weaker inequalities
 \begin{equation*}
 \prod_{i=1}^n
 \left(1-n\frac{\lambda_i(K_s,\Lambda)}{2}\right) \leq
 \frac{\LE(K,\Lambda)}{\vol(K)}\det(\Lambda)\leq \prod_{i=1}^n \left(1+n\frac{\lambda_i(K_s,\Lambda)}{2}\right).
\end{equation*}
Observe that the upper bound together with Minkowski's upper bound~\eqref{eq:secondminimumgen} give
\begin{equation*}
            \LE(K,\Lambda) \leq
            \left(\frac{2}{\lambda_1(K,\Lambda)}+n\right)\cdots \left(\frac{2}{\lambda_n(K,\Lambda)}+n\right),
\end{equation*}
which in turn via~\eqref{eq:vollattice} implies Minkowski's upper bound~\eqref{eq:secondminimumgen}~\cite{FreyerLucas2022}.

A different point of view on Conjecture~\ref{conj:discreteMinkowski} is
given by Ehrhart theory from Discrete Geometry.
By the monotonicity of the successive minima it suffices to prove the
conjecture instead of for $K\in\Kn$ for the associated lattice polytope
$P:=\conv(K\cap\Lambda)$. A polytope is called a~lattice polytope
(with respect to a~lattice $\Lambda$) if all its vertices are lattice
points of $\Lambda$. According to a~result due to Ehrhart~\cite{Ehrhart1962} we have for $k\in\N$
\begin{equation*}
   \LE(k\,P,\Lambda) =\sum_{i=0}^n \LE_i(P,\Lambda) k^i,
\end{equation*}
which is also known as the Ehrhart-polynomial. The coefficients
$\LE_i(P,\Lambda)$ have been the subject of intensive investigations over
the last decades (see~\cite{BeckRobins2015}) and (at least) three of them have a~very clear
geometric meaning
\begin{equation*}
      \LE_n(P,\Lambda)=\frac{\vol(P)}{\det(\Lambda)},\quad
      \LE_{n-1}(P,\Lambda)=\frac{1}{2}\sum_{i=1}^m
      \frac{\vol_{n-1}(F_i)}{\det(\aff(F_i)\cap\Lambda)},\quad \LE_0(P,\Lambda)=1,
\end{equation*}
where $F_1,\dots,F_m$ are the facets of $P$, and $\LE_0(P,\Lambda)$
corresponds to the Euler-characteristic of $P$.
Hence, one may also
ask for relations of the other coefficients (apart from the volume) to
the successive minima. In~\cite{HenkSchuermannWills2005} it was
shown for $o$-symmetric lattice polytopes
\begin{equation}\label{eq:discretevolsur}
   \frac{\LE_{n-1}(P,\Lambda)}{\LE_n(P,\Lambda)}\leq \frac{1}{2}\sum_{i=1}^n \lambda_i(P,\Lambda).
 \end{equation}
 The inequality is best possible, e.g., for $C_n$ and $\dual{C_n}$,
 and maybe regarded as a~discrete counterpart to~\eqref{eq:survol} for $o$-symmetric lattice polytopes. For
 generalizations to not necessarily symmetric polytopes or lattice polytopes
 having their centroid at the origin we refer to~\cite{HenkHenzeHernandezCifre2016}. Here we want to
 point out one nice feature of the above inequality for symmetric
 polytopes; together with Minkowski's upper bound it gives the best
 possible inequality
\begin{equation*}
        \LE_{n-1}(P,\Lambda) \leq \sum_{j=1}^n \prod_{i\ne j}
   \frac{2}{\lambda_i(P,\Lambda)}.
 \end{equation*}
 Hence, in view of Minkowski's upper bound on the volume it is
 tempting to conjecture that $\LE_i(P,\Lambda)$ is bounded by the
 $i$-th elementary symmetric function of the successive minima. This,
 however, fails already for $i=n-2$ as shown in~\cite[Proposition
 1.1]{BeyHenkHenzeEtAl2011}, but, on the positive side it is true for
 special lattice polytopes, as parallelepipeds and symmetric
 lattice-face polytopes (see~\cite{BeyHenkHenzeEtAl2011}).

The problem of counting lattice points inside a~convex body may
  also be considered from the more general point of view of covering
 lattice points by a~minimum number of $k$-dimensional affine
 subspaces. For those types of covering problems we refer to~\cite{BalkoCibulkaValtr2019, BaranyHarcosPachEtAl2001,
   BezdekLitvak2009, BrassKnauer2003, FukshanskyHsu2023}
 and here, as an appetizer we only state the following result by Balko et
 al.~\cite[Theorem 2.5]{BalkoCibulkaValtr2019}: Let $K\in\cKn$
 containing $n$ linearly independent lattice points of $\Lambda$ and
 let $1\leq k\leq n-1$. Then up to constants depending on $k$ and $n$
 the lattice points of $K\cap\Lambda$ can be covered by
\begin{equation*}
 \left(\lambda_{k+1}(K,\Lambda)\cdots\lambda_n(K,\Lambda)\right)^{-1}
 \end{equation*}
$k$-dimensional affine subspaces and the bound is optimal -- up to constants depending on $k$ and $n$.

\section{An application: Bombieri-Vaaler extension of Siegel's lemma}\label{sec:bombierivaalersiegel}

Let $A=(a_{ij})\in \Z^{m\times n}$, $m<n$, be an integral matrix of rank $m$.
Consider the system of homogeneous linear equations
\begin{equation}
A{\boldsymbol x}={\boldsymbol 0}\,.
\label{sl_s}
\end{equation}
Since $m<n$, the system~\eqref{sl_s} has a~non--trivial solution in integers. If the entries of $A$
are relatively small integers, then it is reasonable to expect that there will be a~solution in relatively small integers.
This principle was applied by Thue in~\cite{Thue1909} to a~problem from Diophantine approximations.
Siegel~\cite[Bd. I, p. 213, Hilfssatz]{Siegel1929}
was the first to state this idea formally.

Let us denote by $\|A\|_\infty$ the maximum absolute value of an entry of $A$, that is $\|A\|_\infty=\max_{ij} |a_{ij}|$.
Following Siegel's work, one can obtain the following result (included with proof in Schmidt~\cite{Schmidt1991}).
\begin{theorem}[Siegel's Lemma]
The system~\eqref{sl_s} has a solution ${\boldsymbol x}\in \Z^n$ with
\begin{equation}
0<\|{\boldsymbol x}\|_\infty< 1 + (n\|A\|_\infty)^{m/(n-m)}\,.
\label{sl_sl_f}
\end{equation}
\label{sl}
\end{theorem}
Notably, the exponent $m/(n-m)$ on the right hand side of (\ref{sl_sl_f}) is optimal.
Siegel's lemma type results have been motivated by their numerous applications in number theory
(see e.g.~\cite{BombieriVaaler1983,Schmidt1991,Fukshansky2006a}). In more recent years, new applications have been developed, in particular in mathematical optimization~\cite{AlievDeLoeraEisenbrandEtAl2018, AlievDeLoeraOertelEtAl2017}.

To establish a~link between Siegel's lemma and successive minima, we follow the work of Bombieri and Vaaler~\cite{BombieriVaaler1983}. They have proved, by using geometry of numbers, the following advanced version of Siegel's lemma.
\begin{theorem}\label{BV_theo}
The system~\eqref{sl_s} has $n-m$
linearly independent solutions
${\boldsymbol x}_1, \ldots, {\boldsymbol x}_{n-m}$ in $\Z^n$, with
\begin{equation*}
\prod_{i=1}^{n-m}\|{\boldsymbol x}_i\|_{\infty}
\le \frac{\sqrt{\det(AA^{T})}}{\gcd(A)},
\end{equation*}
where $\gcd(A)$ is the greatest common
divisor of all $m\times m$ subdeterminants of $A$.
\end{theorem}

Recall that $C_n=[-1,1]^n$ and let $\ker(A)=\{{\boldsymbol x}\in \R^n : A{\boldsymbol x}={\boldsymbol 0}\}$. Consider the section $S(A)=C_n\cap \ker(A)$ of the cube $C_n$ and the lattice $\Lambda(A)=\Z^n \cap \ker(A)$. The lattice $\Lambda(A)$ has determinant $\det(\Lambda(A))=\sqrt{\det(AA^{T})}/\gcd(A)$. The $(n-m)$-dimensional subspace $\ker(A)$ can be considered as a~usual Euclidean $(n-m)$-dimensional space. This immediately extends the definition of successive minima to $o$-symmetric bounded convex sets with nonempty relative interior in $\ker(A)$ and $(n-m)$-dimensional lattices in $\ker(A)$. Theorem~\ref{BV_theo} is an immediate corollary of the following result.
\begin{theorem}\label{BV_theo_sm} Let $A\in \Z^{m\times n}$, $m<n$, be an integral matrix of rank $m$.
Then the inequality
\begin{equation}
\prod_{i=1}^{n-m} \lambda_i(S(A), \Lambda(A)) \le \det(\Lambda(A))
\end{equation}
holds.
\end{theorem}

\begin{proof}
By a~result of Vaaler~\cite{Vaaler1979}, we have $\vol_{n-m}(S(A))\ge 2^{n-m}$. Hence, Minkowski's theorem on successive minima in the form~\eqref{eq:secondminimumgenlat} gives
\begin{equation*}
\prod_{i=1}^{n-m} \lambda_i(S(A), \Lambda(A)) \le \frac{2^{n-m}\det(\Lambda(A))}{\vol_{n-m}(S(A))}\le \det(\Lambda(A))\,.
\qedhere
\end{equation*}
\end{proof}

In what follows, we will focus on the special case $m=1$, that is when $A$ is just an $n$-dimensional nonzero row vector.
Theorem~\ref{BV_theo} implies that for every nonzero vector
${\boldsymbol a}$ in $\Z^{n}$, $n\ge 2$, there exists a~vector ${\boldsymbol x}$
in $\Z^{n}$, such that
\begin{equation}\label{knapsack_bound}
\ip{\boldsymbol a}{\boldsymbol x}=0\,,\;\;\;
0<\|{\boldsymbol x}\|_{\infty}^{n-1}\le \sqrt{n}\|{\boldsymbol a}\|_\infty\,.
\end{equation}

The exponent $n-1$ in the latter bound is optimal.
Let us define
\begin{equation*}
c(n)=\sup_{{\boldsymbol a}\in\Z^{n}\setminus\{0\}}
\inf_{
\begin{array}
{c}
\scriptstyle{\boldsymbol x}\in\Z^{n}\setminus\{0\}\\
\scriptstyle\ip{\boldsymbol a}{\boldsymbol x}=0
\end{array}
}
\frac{\|{\boldsymbol x}\|_\infty^{n-1}}{\|{\boldsymbol a}\|_\infty}\,.
\end{equation*}
That is $c(n) $ is the optimal constant in the bound~\eqref{knapsack_bound}.

It is easy to see that $c(2)=1$. Further, the equality $c(3)=4/3$
is implicit in~\cite{ChaladusSchinzel1991}. Namely, the inequality
$c(3)\le 4/3$ is contained in~\cite[Lemma 4]{ChaladusSchinzel1991},
while the inequality $c(3)\ge 4/3$ is a~consequence
of~\cite[Lemma 7]{ChaladusSchinzel1991}.
We have also $c(4)=27/19$.
The inequality $c(4)\ge 27/19$ was proved by
Chaladus in~\cite{Chaladus1992} and its counterpart $c(4)\le 27/19$ was obtained in~\cite{Aliev2001} (see also~\cite{Schinzel2002}).
For $n>4$, the exact values of the constants $c(n)$ remain unknown.

In this vein, Schinzel~\cite{Schinzel2002} proved
the following general result that gives a~geometric interpretation for the constant $c(n)$. Given $K\in \cKn$ we denote by $\Delta(K)$ its {critical determinant}, defined as
\begin{equation*}
\Delta(K)=\min\{\det(\Lambda): 2\Lambda \mbox{ is a~packing lattice for } K\}\,.
\end{equation*}
In terms of the density $\delta(K)$ (see~\eqref{def:density}), we have
\begin{equation*}
\Delta(K)=\frac{\vol(K)}{2^n \delta(K)}\,.
\end{equation*}

\begin{theorem}
For $n\ge 3$
\begin{equation*}
c(n)=\sup \Delta(
H_{\alpha_1,\ldots,\alpha_{n-3}}
)^{-1}\,,
\end{equation*}
where
${H}_{\alpha_1,\ldots,\alpha_{n-3}}$
is a~generalised hexagon given by
\begin{equation*}
{H}_{\alpha_1,\ldots,\alpha_{n-3}}=\left\{\vx\in\R^{n-1}: \|{\boldsymbol
    x}\|_\infty\leq 1, \;
\left |\sum_{i=1}^{n-3}\alpha_ix_i+x_{n-2}+x_{n-1}\right |\le 1\right\}
\end{equation*}
and the supremum is taken over all rational numbers $\alpha_1,\ldots,\alpha_{n-3}$ in the
interval $(0,1]$.
\end{theorem}

Based on the values of $c(n)$ for $n\le 4$, the following conjecture was proposed in~\cite{Aliev2008}.
\begin{conjecture}
The equality
\begin{equation*}
c(n)=\Delta(H_{1,\ldots,1})^{-1}
\end{equation*}
holds. Here $H_{1,\ldots,1}$ is a~generalised hexagon in $\R^{n-1}$.
\end{conjecture}

From the perspective of Theorems~\ref{BV_theo}--\ref{BV_theo_sm} and the successive minima, it is natural to consider for $n\ge 2$ the constant
\begin{equation}\label{constant_s}
s(n)=\sup_{{\boldsymbol a}\in\Z^{n}\setminus\{0\}}
\inf_{{\boldsymbol x}_1, \ldots, {\boldsymbol x}_{n-1}}\frac{\prod_{i=1}^{n-1}\|{\boldsymbol x}_i\|_\infty}{\|{\boldsymbol a}\|_\infty}\,,
\end{equation}
where the infimum is taken over all linearly independent integer vectors $ {\boldsymbol x}_1, \ldots, {\boldsymbol x}_{n-1}\in\Z^n$ such that
$\ip{\boldsymbol a}{{\boldsymbol x}_1}=\cdots=\ip{\boldsymbol a}{{\boldsymbol x}_{n-1}}=0$.

The bound in Theorem~\eqref{BV_theo} immediately implies
\begin{equation}\label{old_bound_1}
s(n)\le\sqrt{n}\,.
\end{equation}

In~\cite{Aliev2008}, the constant $s(n)$ was estimated as
\begin{equation}
s(n)\le\sigma_{n}^{-1}\,,\label{tozhe_nasha}
\end{equation}
where $\sigma_{n}$ is the {sinc} integral
\begin{equation*}
\sigma_{n}=\frac{2}{\pi}\int_0^\infty\left(\frac{\sin
t}{t}\right)^{n} dt\,.
\end{equation*}
The bound~\eqref{tozhe_nasha}
asymptotically
improves on~\eqref{old_bound_1} with factor $\sqrt{\pi/6}$.
The numbers $\sigma_{n}$ are rational, the sequences of numerators and denominators of $\sigma_{n}/2$ can be
found in~\cite{Sloane} (sequences A049330 and A049331).

Clearly, $s(2)=c(2)=1$. In this section we prove the following result, mentioned without proof in~\cite[Remark 1 (ii)]{Aliev2008}.
\begin{theorem}\label{small_s}
For $n\in\{3,4\}$ we have $s(n)=c(n)$. That is
\begin{equation*}
s(3)=\frac{4}{3} \mbox{ and } s(4)=\frac{27}{19}\,.
\end{equation*}
\end{theorem}

\begin{proof}
Observe that for any nonzero ${\boldsymbol a}\in \Z^{n}$
\begin{equation*}
\inf_{
\begin{array}
{c}
\scriptstyle{\boldsymbol x}\in\Z^{n}\setminus\{0\}\\
\scriptstyle\ip{\boldsymbol a}{\boldsymbol x}=0
\end{array}
}\frac{\|{\boldsymbol x}\|_\infty^{n-1}}{\|{\boldsymbol a}\|_\infty}
\le \inf_{{\boldsymbol x}_1, \ldots, {\boldsymbol x}_{n-1}}\frac{\prod_{i=1}^{n-1}\|{\boldsymbol x}_i\|_\infty}{\|{\boldsymbol a}\|_\infty}\,,
\end{equation*}
where the latter infimum is taken over the same set as in~\eqref{constant_s}. Hence for all $n\ge 2$
we have
\begin{equation}\label{c_less_s}
c(n)\le s(n)\,.
\end{equation}

In view of~\eqref{c_less_s}, it is sufficient to show that $s(3)\le 4/3$ and $s(4)\le 27/19$. To achieve this goal, we will first express $s(n)$ in terms of successive minima.
Given a~nonzero ${\boldsymbol a}\in \Z^{n}$, let $\ker({\boldsymbol a})=\{{\boldsymbol x}\in \R^{n}: \ip{\boldsymbol a}{\boldsymbol x}=0\}$, and consider the lattice
\begin{equation*}
\Lambda({\boldsymbol a})= \Z^{n}\cap \ker({\boldsymbol a})\,,
\end{equation*}
and the $(n-1)$-dimensional section
\begin{equation*}
S({\boldsymbol a})=C_{n}\cap\ker({\boldsymbol a})
\end{equation*}
of the cube $C_{n}$.
Then, by the definition of successive minima, we have
\begin{equation}\label{s_via_sm}
s(n)=\sup_{{\boldsymbol a}\in\Z^{n}\setminus\{0\}} \frac{\prod_{i=1}^{n-1} \lambda_i(S({\boldsymbol a}), \Lambda({\boldsymbol a}))}{\|{\boldsymbol a}\|_\infty}\,.
\end{equation}

Given $K\in {\mathcal K}_{os}^n$, its {anomaly} $a(K)$ is defined as
\begin{equation*}
a(K)=\sup_{\Lambda\in \Lat^n}\frac{\Delta(K) \prod_{i=1}^n \lambda_i(K, \Lambda)}{\det(\Lambda)}\,.
\end{equation*}
The Problem~\ref{conj:davenport} of Davenport in terms of the anomaly is asking whether $a(K)=1$.
Woods~\cite{Woods1956} proved that $a(K)=1$ holds in dimension up to three. As above, notice that the hyperplane $\ker(\va)$ can be considered as a~usual Euclidean $(n-1)$-dimensional space. This immediately extends the definition of the critical determinant to $o$-symmetric convex sets with nonempty relative interior in $\ker({\boldsymbol a})$.
Hence, for $n\le 4$ we have
\begin{equation}\label{prod_2}
\prod_{i=1}^{n-1} \lambda_i(S({\boldsymbol a}), \Lambda({\boldsymbol a})) \le \frac{\det(\Lambda(\va))}{\Delta(S({\boldsymbol a}))}=\frac{\|{\boldsymbol a}\|_2}{\gcd({\boldsymbol a})\Delta(S({\boldsymbol a}))}\,.
\end{equation}

We may assume without loss of generality that ${\boldsymbol a}$ does not have zero entries. Otherwise, we can replace $n$ with $n-1$. Hence, without loss of generality, we may assume that $0<a_1\le\cdots\le a_{n}$. Consider the projection $\pi: \R^{n}\rightarrow \R^{n-1}$ that forgets
the last coordinate. Since all entries of ${\boldsymbol a}$ are positive, the mapping $\pi$ restricted to $\ker({\boldsymbol a})$ is bijective. Let $K({\boldsymbol a})=\pi(S({\boldsymbol
  a}))$. One can write $K({\boldsymbol a})$ as follows. Given a sequence of rational numbers $0<\alpha_1\le \cdots\le \alpha_{n-1}\le 1$, let
\begin{equation*}
K_{\alpha_1, \ldots, \alpha_{n-1}}=\{{\boldsymbol x}\in\R^{n-1}: \|{\boldsymbol x}\|_\infty\le 1, |\alpha_1x_1+\cdots+\alpha_{n-1} x_{n-1}|\le 1\}\,.
\end{equation*}
Then $K({\boldsymbol a})=K_{\alpha_1, \ldots, \alpha_{n-1}}$ with
\begin{equation}
\alpha_1=\frac{a_1}{a_{n}}, \ldots, \alpha_{n-1}=\frac{a_{n-1}}{a_{n}}\,.
\label{set_alpha_i}
\end{equation}

For any $(n-1)$-dimensional lattice $\Lambda\subset \ker({\boldsymbol a})$, the lattice $2\Lambda$ is a~packing lattice for $S({\boldsymbol a})$ if and only if $2\pi(\Lambda)$ is a~packing lattice for $K({\boldsymbol a})$. Hence
\begin{equation*}
\Delta(S({\boldsymbol a}))=\Delta(K({\boldsymbol a})) \frac{\|{\boldsymbol a}\|_2}{\|{\boldsymbol a}\|_{\infty}}
\end{equation*}
and, by~\eqref{prod_2},
\begin{equation*}
\prod_{i=1}^{n-1} \lambda_i(S({\boldsymbol a}), \Lambda({\boldsymbol a})) \le \frac{\|{\boldsymbol a}\|_\infty}{\gcd({\boldsymbol a})\Delta(K({\boldsymbol a}))}\,.
\end{equation*}
Consequently, by~\eqref{s_via_sm}, we obtain the inequality
\begin{equation}\label{s_via_Deltas}
s(n)\le \sup_{{\boldsymbol a}\in\Z^{n}\setminus\{0\}}\frac{1}{\gcd(\va)\Delta(K({\boldsymbol a}))}\,.
\end{equation}

The main tool of the proof of Theorem~\ref{small_s} is the following lemma.
\begin{lemma}\label{smaller_section} For any $n\ge 3$ and any rational numbers $0<\alpha_1\le \cdots \le\alpha_{n-1}\le 1$, the following statements hold:
\begin{itemize}
\item[(i)] If $n=3$ then
\begin{equation*}
K_{1,1}\subset K_{\alpha_1, \alpha_2}\,.
\end{equation*}
\item[(ii)] If $n>3$ then there exists rational numbers
$0<\beta_1\le \cdots \le\beta_{n-3}\le 1$ such that
\begin{equation*}
K_{\beta_1, \ldots, \beta_{n-3},1,1}\subset K_{\alpha_1, \ldots, \alpha_{n-1}}\,.
\end{equation*}
\end{itemize}
\end{lemma}

This result was originally proved for $n\le 4$ in~\cite{Aliev2001}, and, subsequently, for all $n$ in~\cite{Schinzel2002}. We include a~proof for completeness.

\begin{proof}
We start with part (i). For ${\boldsymbol x}\in K_{1,1}$, we have
\begin{equation}\label{hexagon2}
\begin{aligned}
|x_1|\le 1, |x_2|\le 1, \left | x_{1} +x_2 \right|\le 1\,.
\end{aligned}
\end{equation}
Multiplying the inequality $|x_2|\le 1$ by $\alpha_2/\alpha_{1}-1$ and adding it to the last inequality in~\eqref{hexagon2} we obtain
the inequality $ \left | \alpha_1 x_{1} + \alpha_2 x_2 \right|\le \alpha_2\le 1$. Hence ${\boldsymbol x}\in K_{\alpha_1, \alpha_2}$.

For part (ii) take
\begin{equation*}
\beta_1=\frac{\alpha_1}{\alpha_{n-2}},\ldots, \beta_{n-3}=\frac{\alpha_{n-3}}{\alpha_{n-2}}\,.
\end{equation*}
Then if ${\boldsymbol x}\in K_{\beta_1, \ldots, \beta_{n-3},1,1}$, we have
\begin{equation}\label{hexagon}
\begin{aligned}
|x_i|\le 1, \; i\in\{1,\ldots, n-1\}\,,\\
\left | \sum_{i=1}^{n-3} \frac{\alpha_i}{\alpha_{n-2}}x_i +x_{n-2} +x_{n-1} \right|\le 1\,.
\end{aligned}
\end{equation}
Multiplying the inequality $|x_{n-1}|\le 1$ by $\alpha_{n-1}/\alpha_{n-2}-1$ and adding it to the last inequality in~\eqref{hexagon}, we obtain the inequality
\begin{equation*}
 \frac{\alpha_{n-1}}{\alpha_{n-2}} \ge \left| \sum_{i=1}^{n-3} \frac{\alpha_i}{\alpha_{n-2}} x_i +x_{n-2} + x_{n-1} \right | + \left(
 \frac{\alpha_{n-1}}{\alpha_{n-2}}-1\right)|x_{n-1}|\,.
\end{equation*}
Multiplying both sides by $\alpha_{n-2}$ we obtain
\begin{equation*}
\begin{aligned}
1\ge \alpha_{n-1} &\ge \left | \sum_{i=1}^{n-3} \alpha_i x_i +\alpha_{n-2} x_{n-2} + \alpha_{n-2} x_{n-1} \right | + \left(
 \alpha_{n-1}-\alpha_{n-2}\right)|x_{n-1}|\\
&\ge
|\alpha_1x_1+\cdots+\alpha_{n-1} x_{n-1}|\,.
\end{aligned}
\end{equation*}
Hence ${\boldsymbol x}\in K_{\alpha_1, \ldots, \alpha_{n-1}}$.
\end{proof}

For the rest of the proof we choose the numbers $\alpha_i$ as in~\eqref{set_alpha_i}.
Suppose first that $n=3$. Since $K_{1,1}\subset K_{\alpha_1,
  \alpha_2}$ we have by Lemma~\ref{smaller_section} part (i)
\begin{equation*}
\Delta(K({\boldsymbol a}))=\Delta(K_{\alpha_1, \alpha_2})\ge \Delta(K_{1, 1})\,.
\end{equation*}
Further, since the hexagon $\Delta(K_{1, 1})$ is a~space filling convex body (we refer the reader to~\cite[Section 20.4]{GruberLekkerkerker1987} for details), we have $\delta(K_{1, 1})=1$ and, consequently,
\begin{equation*}
\Delta(K_{1, 1})=\frac{\vol(K_{1, 1})}{4}=\frac{3}{4}\,.
\end{equation*}
The bound~\eqref{s_via_Deltas} completes the proof in this case.

It remains to consider the case $n=4$. By Lemma~\ref{smaller_section} part (ii) there exists a~rational $0<\beta\le 1$ such that $K_{\beta, 1,1}\subset K_{\alpha_1, \alpha_2,\alpha_3}$. A result of Whitworth~\cite{Whitworth1948} implies that $\Delta(K_{\beta,1,1})$ equals
\begin{equation}
\left \{
\begin{array}
{ll}
3/4\,,& 0\le \beta<1/2\,,\\
-(\beta^2 +3\beta-24+1/\beta)/27\,,& 1/2\le\beta\le 1\,.
\end{array}
\right.
\label{sl_det}
\end{equation}
Hence $\Delta(K_{\beta,1,1})$ takes the minimum in the interval $[0,1]$ at $\beta=1$ and
\begin{equation*}
\Delta(K({\boldsymbol a}))=\Delta(K_{\alpha_1, \alpha_2,\alpha_3})\ge \Delta(K_{1,1,1})=19/27.
\end{equation*}
The bound~\eqref{s_via_Deltas} completes the proof of Theorem~\ref{small_s}.
\end{proof}

\subsection*{Acknowledgements}
The authors thank Matthias Schymura for many valuable comments on an earlier version, and the anonymous referee for detailed and constructive comments.


\EditInfo{April 4, 2023}{April 27, 2023}{Camilla Hollanti and Lenny Fukshansky}


{\small
\begin{thebibliography}{00}

\bibitem{Aliev2001}
I.~Aliev.
\newblock {\em On decomposition of integer vectors}.
\newblock PhD thesis, Institute of Mathematics PAN, Warsaw, 2001.

\bibitem{Aliev2008}
I.~Aliev.
\newblock Siegel's lemma and sum-distinct sets.
\newblock {\em Discrete Comput. Geom.}, 39(1-3):59--66, 2008.

\bibitem{AlievDeLoeraEisenbrandEtAl2018}
I.~Aliev, J.~A. De~Loera, F.~Eisenbrand, T.~Oertel, and R.~Weismantel.
\newblock The support of integer optimal solutions.
\newblock {\em SIAM J. Optim.}, 28(3):2152--2157, 2018.

\bibitem{AlievDeLoeraOertelEtAl2017}
I.~Aliev, J.~A. De~Loera, T.~Oertel, and C.~O'Neill.
\newblock Sparse solutions of linear {D}iophantine equations.
\newblock {\em SIAM J. Appl. Algebra Geom.}, 1(1):239--253, 2017.

\bibitem{PaivaBalacheffTzanev2016}
J.~C. \'{A}lvarez Paiva, F.~Balacheff, and K.~Tzanev.
\newblock Isosystolic inequalities for optical hypersurfaces.
\newblock {\em Adv. Math.}, 301:934--972, 2016.

\bibitem{BalkoCibulkaValtr2019}
M.~Balko, J.~Cibulka, and P.~Valtr.
\newblock Covering lattice points by subspaces and counting point-hyperplane
  incidences.
\newblock {\em Discrete Comput. Geom.}, 61(2):325--354, 2019.

\bibitem{Banaszczyk1993a}
W.~Banaszczyk.
\newblock New bounds in some transference theorems in the geometry of numbers.
\newblock {\em Math. Ann.}, 296(4):625--635, 1993.

\bibitem{Banaszczyk1996}
W.~Banaszczyk.
\newblock Inequalities for convex bodies and polar reciprocal lattices in
  {$\mathbf{R}^n$}. {II}. {A}pplication of {$K$}-convexity.
\newblock {\em Discrete Comput. Geom.}, 16(3):305--311, 1996.

\bibitem{BanaszczykLitvakPajorEtAl1999}
W.~Banaszczyk, A.~E. Litvak, A.~Pajor, and S.~J. Szarek.
\newblock The flatness theorem for nonsymmetric convex bodies via the local
  theory of {B}anach spaces.
\newblock {\em Math. Oper. Res.}, 24(3):728--750, 1999.

\bibitem{BaranyHarcosPachEtAl2001}
I.~B\'{a}r\'{a}ny, G.~Harcos, J.~Pach, and G.~Tardos.
\newblock Covering lattice points by subspaces.
\newblock {\em Period. Math. Hungar.}, 43(1-2):93--103, 2001.

\bibitem{BeckRobins2015}
M.~Beck and S.~Robins.
\newblock {\em Computing the continuous discretely}.
\newblock Undergraduate Texts in Mathematics. Springer, New York, second
  edition, 2015.
\newblock Integer-point enumeration in polyhedra, With illustrations by David
  Austin.

\bibitem{BermanBerndtsson2017}
R.~J. Berman and B.~Berndtsson.
\newblock The volume of {K}\"{a}hler-{E}instein {F}ano varieties and convex
  bodies.
\newblock {\em J. Reine Angew. Math.}, 723:127--152, 2017.

\bibitem{BetkeHenkWills1993}
U.~Betke, M.~Henk, and J.~M. Wills.
\newblock Successive-minima-type inequalities.
\newblock {\em Discrete Comput. Geom.}, 9(2):165--175, 1993.

\bibitem{BeyHenkHenzeEtAl2011}
C.~Bey, M.~Henk, M.~Henze, and E.~Linke.
\newblock Notes on lattice points of zonotopes and lattice-face polytopes.
\newblock {\em Discrete Math.}, 311(8-9):634--644, 2011.

\bibitem{BezdekLitvak2009}
K.~Bezdek and A.~E. Litvak.
\newblock Covering convex bodies by cylinders and lattice points by flats.
\newblock {\em J. Geom. Anal.}, 19(2):233--243, 2009.

\bibitem{BombieriGubler2006}
E.~Bombieri and W.~Gubler.
\newblock {\em Heights in {D}iophantine geometry}, volume~4 of {\em New
  Mathematical Monographs}.
\newblock Cambridge University Press, Cambridge, 2006.

\bibitem{BombieriVaaler1983}
E.~Bombieri and J.~Vaaler.
\newblock On {S}iegel's lemma.
\newblock {\em Invent. Math.}, 73(1):11--32, 1983.

\bibitem{BrassKnauer2003}
P.~Brass and C.~Knauer.
\newblock On counting point-hyperplane incidences.
\newblock volume~25, pages 13--20. 2003.
\newblock Special issue on the European Workshop on Computational
  Geometry--CG01 (Berlin).

\bibitem{BrazitikosGiannopoulosValettasEtAl2014}
S.~Brazitikos, A.~Giannopoulos, P.~Valettas, and B.-H. Vritsiou.
\newblock {\em Geometry of isotropic convex bodies}, volume 196 of {\em
  Mathematical Surveys and Monographs}.
\newblock American Mathematical Society, Providence, RI, 2014.

\bibitem{BuragoZalgaller1988}
Y.~D. Burago and V.~A. Zalgaller.
\newblock {\em Geometric inequalities}, volume 285 of {\em Grundlehren der
  mathematischen Wissenschaften [Fundamental Principles of Mathematical
  Sciences]}.
\newblock Springer-Verlag, Berlin, 1988.
\newblock Translated from the Russian by A. B. Sosinski\u{\i}, Springer Series
  in Soviet Mathematics.

\bibitem{CamposHintumMorrisEtAl2022}
M.~Campos, P.~van Hintum, R.~Morris, and M.~Tiba.
\newblock Towards hadwiger's conjecture via bourgain slicing, 2022.
\newblock arXiv:2206.11227.

\bibitem{Cassels1971}
J.~W.~S. Cassels.
\newblock {\em An introduction to the geometry of numbers}.
\newblock Die Grundlehren der mathematischen Wissenschaften, Band 99.
  Springer-Verlag, Berlin-New York, 1971.
\newblock Second printing, corrected.

\bibitem{Chaladus1992}
S.~Cha{\l}adus.
\newblock On the densest lattice packing of centrally symmetric octahedra.
\newblock {\em Math. Comp.}, 58(197):341--345, 1992.

\bibitem{ChaladusSchinzel1991}
S.~Cha{\l}adus and A.~Schinzel.
\newblock A decomposition of integer vectors. {II}.
\newblock {\em Pliska Stud. Math. Bulgar.}, 11:15--23, 1991.

\bibitem{Chang2002}
M.-C. Chang.
\newblock A polynomial bound in {F}reiman's theorem.
\newblock {\em Duke Math. J.}, 113(3):399--419, 2002.

\bibitem{Davenport1946}
H.~Davenport.
\newblock The product of {$n$} homogeneous linear forms.
\newblock {\em Nederl. Akad. Wetensch., Proc.}, 49:822--828 = Indagationes
  Math. 8, 525--531 (1946), 1946.

\bibitem{Eggleston1961}
H.~G. Eggleston.
\newblock Note on a conjecture of {L}. {A}. {S}antal\'{o}.
\newblock {\em Mathematika}, 8:63--65, 1961.

\bibitem{Ehrhart1955a}
E.~Ehrhart.
\newblock Une g\'{e}n\'{e}ralisation du th\'{e}or\`eme de {M}inkowski.
\newblock {\em C. R. Acad. Sci. Paris}, 240:483--485, 1955.

\bibitem{Ehrhart1962}
E.~Ehrhart.
\newblock Sur les poly\`edres rationnels homoth\'{e}tiques \`a {$n$}
  dimensions.
\newblock {\em C. R. Acad. Sci. Paris}, 254:616--618, 1962.

\bibitem{Ehrhart1964}
E.~Ehrhart.
\newblock Une g\'{e}n\'{e}ralisation probable du th\'{e}or\`eme fondamental de
  {M}inkowski.
\newblock {\em C. R. Acad. Sci. Paris}, 258:4885--4887, 1964.

\bibitem{Ehrhart1979a}
E.~Ehrhart.
\newblock Volume r\'{e}ticulaire critique d'un simplexe.
\newblock {\em J. Reine Angew. Math.}, 305:218--220, 1979.

\bibitem{Evertse2019}
J.-H. Evertse.
\newblock Mahler's work on the geometry of numbers.
\newblock {\em Doc. Math.}, Extra Vol.:29--43, 2019.

\bibitem{FejesTothMakai1974}
L.~Fejes~T\'{o}th and E.~Makai, Jr.
\newblock On the thinnest non-separable lattice of convex plates.
\newblock {\em Studia Sci. Math. Hungar.}, 9:191--193 (1975), 1974.

\bibitem{FradeliziMeyerZvavitch2023}
M.~Fradelizi, M.~Meyer, and A.~Zvavitch.
\newblock Volume product, 2023.
\newblock arXiv:2301.06131.

\bibitem{FreyerLucas2022}
A.~Freyer and E.~Lucas.
\newblock Interpolating between volume and lattice point enumerator with
  successive minima.
\newblock {\em Monatsh. Math.}, 198(4):717--740, 2022.

\bibitem{Fukshansky2006a}
L.~Fukshansky.
\newblock Siegel's lemma with additional conditions.
\newblock {\em J. Number Theory}, 120(1):13--25, 2006.

\bibitem{FukshanskyHsu2023}
L.~Fukshansky and A.~Hsu.
\newblock Covering {P}oint-{S}ets with {P}arallel {H}yperplanes and {S}parse
  {S}ignal {R}ecovery.
\newblock {\em Discrete Comput. Geom.}, 69(3):919--930, 2023.

\bibitem{GonzalezMerinoHenze2016}
B.~Gonz\'{a}lez~Merino and M.~Henze.
\newblock A generalization of the discrete version of {M}inkowski's fundamental
  theorem.
\newblock {\em Mathematika}, 62(3):637--652, 2016.

\bibitem{MerinoSchymura2017}
B.~Gonz\'{a}lez~Merino and M.~Schymura.
\newblock On densities of lattice arrangements intersecting every
  {$i$}-dimensional affine subspace.
\newblock {\em Discrete Comput. Geom.}, 58(3):663--685, 2017.

\bibitem{Gruber2007}
P.~M. Gruber.
\newblock {\em Convex and discrete geometry}, volume 336 of {\em Grundlehren
  der mathematischen Wissenschaften [Fundamental Principles of Mathematical
  Sciences]}.
\newblock Springer, Berlin, 2007.

\bibitem{GruberLekkerkerker1987}
P.~M. Gruber and C.~G. Lekkerkerker.
\newblock {\em Geometry of numbers}, volume~37 of {\em North-Holland
  Mathematical Library}.
\newblock North-Holland Publishing Co., Amsterdam, second edition, 1987.

\bibitem{Henk1990}
M.~Henk.
\newblock Inequalities between successive minima and intrinsic volumes of a
  convex body.
\newblock {\em Monatsh. Math.}, 110(3-4):279--282, 1990.

\bibitem{HenkHenzeHernandezCifre2016}
M.~Henk, M.~Henze, and M.~A. Hern\'{a}ndez~Cifre.
\newblock Variations of {M}inkowski's theorem on successive minima.
\newblock {\em Forum Math.}, 28(2):311--325, 2016.

\bibitem{HenkSchuermannWills2005}
M.~Henk, A.~Sch\"{u}rmann, and J.~M. Wills.
\newblock Ehrhart polynomials and successive minima.
\newblock {\em Mathematika}, 52(1-2):1--16 (2006), 2005.

\bibitem{HenkXue2019}
M.~Henk and F.~Xue.
\newblock On successive minima-type inequalities for the polar of a convex
  body.
\newblock {\em Rev. R. Acad. Cienc. Exactas F\'{\i}s. Nat. Ser. A Mat. RACSAM},
  113(3):2601--2616, 2019.

\bibitem{Khinchin1948}
A.~Y. Hin\v{c}in.
\newblock A quantitative formulation of the approximation theory of
  {K}ronecker.
\newblock {\em Izvestiya Akad. Nauk SSSR. Ser. Mat.}, 12:113--122, 1948.

\bibitem{HuangSlomkaTkoczEtAl2022}
H.~Huang, B.~A. Slomka, T.~Tkocz, and B.-H. Vritsiou.
\newblock Improved bounds for {H}adwiger's covering problem via thin-shell
  estimates.
\newblock {\em J. Eur. Math. Soc. (JEMS)}, 24(4):1431--1448, 2022.

\bibitem{IriyehShibata2020}
H.~Iriyeh and M.~Shibata.
\newblock Symmetric {M}ahler's conjecture for the volume product in the
  {$3$}-dimensional case.
\newblock {\em Duke Math. J.}, 169(6):1077--1134, 2020.

\bibitem{KannanLovasz1986}
R.~Kannan and L.~Lov\'{a}sz.
\newblock Covering minima and lattice-point-free convex bodies.
\newblock {\em Ann. of Math. (2)}, 128(3):577--602, 1988.

\bibitem{Klartag2023}
B.~Klartag.
\newblock Logarithmic bounds for isoperimetry and slices of convex sets, 2023.
\newblock arXiv:2303.14938.

\bibitem{Kuperberg2008}
G.~Kuperberg.
\newblock From the {M}ahler conjecture to {G}auss linking integrals.
\newblock {\em Geom. Funct. Anal.}, 18(3):870--892, 2008.

\bibitem{Mahler1938}
K.~Mahler.
\newblock Ein {Minimalproblem} f{\"u}r konvexe {Polygone}.
\newblock {\em Mathematica B, Zutphen}, 7:118--127, 1938.

\bibitem{Mahler1939}
K.~Mahler.
\newblock Ein \"{U}bertragungsprinzip f\"{u}r konvexe {K}\"{o}rper.
\newblock {\em \v{C}asopis P\v{e}st. Mat. Fys.}, 68:93--102, 1939.

\bibitem{Mahler1974}
K.~Mahler.
\newblock Polar analogues of two theorems by {M}inkowski.
\newblock {\em Bull. Austral. Math. Soc.}, 11:121--129, 1974.

\bibitem{Makai1978}
E.~Makai, Jr.
\newblock On the thinnest nonseparable lattice of convex bodies.
\newblock {\em Studia Sci. Math. Hungar.}, 13(1-2):19--27 (1981), 1978.

\bibitem{Malikiosis2012}
R.-D. Malikiosis.
\newblock A discrete analogue for {M}inkowski's second theorem on successive
  minima.
\newblock {\em Adv. Geom.}, 12(2):365--380, 2012.

\bibitem{Malikiosis2013}
R.-D. Malikiosis.
\newblock Lattice-point enumerators of ellipsoids.
\newblock {\em Combinatorica}, 33(6):733--744, 2013.

\bibitem{MayrhoferSchadeWeltge2022}
L.~Mayrhofer, J.~Schade, and S.~Weltge.
\newblock Lattice-free simplices with lattice width {$2d - o(d)$}.
\newblock In {\em Integer programming and combinatorial optimization}, volume
  13265 of {\em Lecture Notes in Comput. Sci.}, pages 375--386. Springer, Cham,
  [2022] \copyright 2022.

\bibitem{MicciancioGoldwasser2002}
D.~Micciancio and S.~Goldwasser.
\newblock {\em Complexity of lattice problems}, volume 671 of {\em The Kluwer
  International Series in Engineering and Computer Science}.
\newblock Kluwer Academic Publishers, Boston, MA, 2002.
\newblock A cryptographic perspective.

\bibitem{MilmanPajor2000}
V.~D. Milman and A.~Pajor.
\newblock Entropy and asymptotic geometry of non-symmetric convex bodies.
\newblock {\em Adv. Math.}, 152(2):314--335, 2000.

\bibitem{MilnorHusemoller1973}
J.~Milnor and D.~Husemoller.
\newblock {\em Symmetric bilinear forms}.
\newblock Ergebnisse der Mathematik und ihrer Grenzgebiete, Band 73.
  Springer-Verlag, New York-Heidelberg, 1973.

\bibitem{Minkowski1893}
H.~Minkowski.
\newblock Extract from a letter to {Mr}. {Hermite}.
\newblock {\em Bull. Sci. Math., II. S{\'e}r.}, 17:24--29, 1893.

\bibitem{Minkowski1896und1910}
H.~Minkowski.
\newblock {\em Geometrie der {Zahlen}}.
\newblock Leipzig: {B}. {G}. {Teubner}, 1896 und 1910.

\bibitem{NillPaffenholz2014}
B.~Nill and A.~Paffenholz.
\newblock On the equality case in {E}hrhart's volume conjecture.
\newblock {\em Adv. Geom.}, 14(4):579--586, 2014.

\bibitem{Peikert2014a}
C.~Peikert.
\newblock A decade of lattice cryptography.
\newblock {\em Found. Trends Theor. Comput. Sci.}, 10(4):i--iii, 283--424,
  2014.

\bibitem{ReisRothvoss2023}
V.~Reis and T.~Rothvoss.
\newblock The subspace flatness conjecture and faster integer programming,
  2023.
\newblock arXiv:2303.14605.

\bibitem{Rudelson2000a}
M.~Rudelson.
\newblock Distances between non-symmetric convex bodies and the
  {$MM^\ast$}-estimate.
\newblock {\em Positivity}, 4(2):161--178, 2000.

\bibitem{Santalo1949}
L.~A. Santal\'{o}.
\newblock An affine invariant for convex bodies of {$n$}-dimensional space.
\newblock {\em Portugal. Math.}, 8:155--161, 1949.

\bibitem{Schinzel2002}
A.~Schinzel.
\newblock A property of polynomials with an application to {S}iegel's lemma.
\newblock {\em Monatsh. Math.}, 137(3):239--251, 2002.

\bibitem{Schmidt1979}
W.~M. Schmidt.
\newblock {\em Diophantine approximation}, volume 785 of {\em Lecture Notes in
  Mathematics}.
\newblock Springer, Berlin, 1980.

\bibitem{Schmidt1991}
W.~M. Schmidt.
\newblock {\em Diophantine approximations and {D}iophantine equations}, volume
  1467 of {\em Lecture Notes in Mathematics}.
\newblock Springer-Verlag, Berlin, 1991.

\bibitem{Schneider2014}
R.~Schneider.
\newblock {\em Convex bodies: the {B}runn-{M}inkowski theory}, volume 151 of
  {\em Encyclopedia of Mathematics and its Applications}.
\newblock Cambridge University Press, Cambridge, expanded edition, 2014.

\bibitem{Schnell1995a}
U.~Schnell.
\newblock Successive minima, intrinsic volumes, and lattice determinants.
\newblock {\em Discrete Comput. Geom.}, 13(2):233--239, 1995.

\bibitem{Siegel1929}
C.~L. Siegel.
\newblock \"{U}ber einige {A}nwendungen diophantischer {A}pproximationen
  [reprint of {A}bhandlungen der {P}reu\ss ischen {A}kademie der
  {W}issenschaften. {P}hysikalisch-mathematische {K}lasse 1929, {N}r. 1].
\newblock In {\em On some applications of {D}iophantine approximations},
  volume~2 of {\em Quad./Monogr.}, pages 81--138. Ed. Norm., Pisa, 2014.

\bibitem{Sloane}
N.~J.~A. Sloane.
\newblock The on-line encyclopedia of integer sequences.
\newblock {\em Notices Amer. Math. Soc.}, 65(9):1062--1074, 2018.

\bibitem{TaoVu2006a}
T.~Tao and V.~Vu.
\newblock {\em Additive combinatorics}, volume 105 of {\em Cambridge Studies in
  Advanced Mathematics}.
\newblock Cambridge University Press, Cambridge, 2006.

\bibitem{Thue1909}
A.~Thue.
\newblock \"{U}ber {A}nn\"{a}herungswerte algebraischer {Z}ahlen.
\newblock {\em J. Reine Angew. Math.}, 135:284--305, 1909.

\bibitem{Tointon2023}
M.~Tointon.
\newblock New bounds in the discrete analogue of {M}inkowski's second theorem,
  2023.
\newblock arXiv:2303.07384.

\bibitem{Vaaler1979}
J.~D. Vaaler.
\newblock A geometric inequality with applications to linear forms.
\newblock {\em Pacific J. Math.}, 83(2):543--553, 1979.

\bibitem{Waerden1956}
B.~L. van~der Waerden.
\newblock Die {R}eduktionstheorie der positiven quadratischen {F}ormen.
\newblock {\em Acta Math.}, 96:265--309, 1956.

\bibitem{Whitworth1948}
J.~V. Whitworth.
\newblock On the densest packing of sections of a cube.
\newblock {\em Ann. Mat. Pura Appl. (4)}, 27:29--37, 1948.

\bibitem{Wills1990b}
J.~M. Wills.
\newblock Minkowski's successive minima and the zeros of a convexity-function.
\newblock {\em Monatsh. Math.}, 109(2):157--164, 1990.

\bibitem{Woods1956}
A.~C. Woods.
\newblock The anomaly of convex bodies.
\newblock {\em Proc. Cambridge Philos. Soc.}, 52:406--423, 1956.

\end{thebibliography}
}
\end{document}